\newif\ifextras
\extrastrue 
\ifextras
\fi 

\documentclass{amsart}

\usepackage{amsmath}
\usepackage{amsthm}
\usepackage{amssymb}
\usepackage{xcolor}
\usepackage{hyperref}
\hypersetup{
    colorlinks=true,
    linkcolor=black,
    filecolor=magenta,      
    urlcolor=blue,
}
\newtheorem*{definition*}{Definition} 

\newtheorem*{theorem*}{Theorem} 
\newtheorem{theorem}{Theorem}
\newtheorem{lemma}{Lemma}
\newtheorem{corollary}{Corollary}

\newcommand{\bHolAlpha}{b^{\rm{hol}}}
\newcommand{\bHarmAlpha}{b^{\rm{harm}}}
\newcommand{\KAlpha}{K_{\alpha}}
\newcommand{\FourierKAlpha}{\left({}_{(y,a,b,z,w)}\KAlpha\right)\hat{\ }}
\newcommand{\KZeroAlpha}{K^0_{\alpha}}
\newcommand{\FourierKZeroAlpha}{\left({}_{(y,b,z)}\KZeroAlpha\right)\hat{\ }}
\newcommand{\RAlpha}{R_{\alpha}}
\newcommand{\FourierRAlpha}{({}_{(y,a,b)}R_\alpha)\hat{\ }}
\newcommand{\BHolAlpha}{B^{\rm{hol}}_{\alpha}}
\newcommand{\BHarmAlpha}{B^{\rm{harm}}_{\alpha}}

\newcommand{\rhoTildeAlpha}{\tilde{\rho}_\alpha}

\newcommand{\myZeroFOne}[2]{{}_0{F}_1\left({{-} \atop #1} \left|\ #2 \frac{}{} \right. \right)}

\renewcommand{\Re}{\mathop{\mathrm{Re}}}
 
\newcommand{\Ric}{\mathop{\mathrm{Ric}}}

\begin{document}

\title{Harmonic Berezin transform on half-space with vertical weights}

\ifextras
\author{Jaroslav Brad\'ik}
\address{Mathematical Institute, Silesian University in Opava, Na Rybn\'i\v{c}ku 1, 74601, Opava, Czech Republic}
\email{jaroslav.bradik@math.slu.cz}
\fi 




\makeatletter
\@namedef{subjclassname@2020}{%
  \textup{2020} Mathematics Subject Classification}
\makeatother

\ifextras
\thanks{Research supported by GA \v{C}R grant no. 21-27941S}
\fi

\begin{abstract}
We consider weighted harmonic Bergman spaces on upper half-space with weights depending only on the vertical coordinate. In these settings, we give full asymptotic expansion of weighted harmonic Bergman kernel as well as full asymptotic expansion of the harmonic Berezin transform for functions depending only on the vertical coordinate. Both expansions utilize known holomorphic case for the Siegel domain.
\end{abstract}

\maketitle


\section{Introduction}

Consider the weighted \emph{harmonic Bergman space} $\bHarmAlpha(\mathbb{H}^n, \omega^\alpha)$ consisting of harmonic functions on the half-space $\mathbb{H}^{n} = \{ (x, y) \in \mathbb{R}^{n-1} \times \mathbb{R} : y > 0 \}$ square integrable with respect to the measure (weight)
\begin{equation*}
d\omega^\alpha = \omega(x, y)^\alpha\ dx\ dy,\quad (x, y) \in \mathbb{H}^{n},
\end{equation*}
where $\omega(x, y)$ is a positive continuous function (weight function) on $\mathbb{H}^n$, $\alpha \geq 0$, and $dx\ dy$ the Lebesgue measure on $\mathbb{H}^n$. It is known that $\bHarmAlpha(\mathbb{H}^n, \omega^\alpha)$ is reproducing kernel Hilbert space possessing uniquely determined function $\RAlpha: \mathbb{H}^n \times \mathbb{H}^n \to \mathbb{R}$, called \emph{harmonic Bergman kernel}, with the reproducing property 
\begin{equation*}
f(a, b) = \int_{\mathbb{H}^n} f(x, y) \RAlpha(a, b; x, y)\ \omega(x, y)^\alpha\ dx\ dy, \quad \forall f \in \bHarmAlpha(\mathbb{H}^n, \omega^\alpha).
\end{equation*}
For more details on the harmonic Bergman spaces and their kernels, see for example \cite{Axler2001} for the unweighted case $\omega^\alpha \equiv 1$. Further, consider the integral transform of bounded function $f$ on $\mathbb{H}^n$ of the form
\begin{equation}\label{eq:harmBerezinTr}
(\BHarmAlpha f)(a, b) = \frac{1}{\RAlpha(a, b; a, b)} \int_{\mathbb{H}^n} f(x, y) \RAlpha(a, b; x, y)^2\ \omega(x, y)^\alpha\ dx\ dy,
\end{equation}
called the \emph{harmonic Berezin transform}.

The above originates from the Berezin-Toeplitz quantization on the spaces of holomorphic, rather than harmonic, functions on some open set $\Omega \subset \mathbb{C}^n$, which has many applications in mathematical physics, see \cite{Berezin1974, MRBerezin1975, Bordemann1994, Coburn1992, Karabegov2001} for quantization on K\"{a}hler manifolds. It has been shown in \cite{Englis2006} that Berezin-Toeplitz quantization cannot be achieved on the harmonic Bergman spaces. Nevertheless, it appears that the harmonic Berezin transform still possesses an asymptotic behaviour but there are no known conditions on the weight to guarantee the asymptotic. In \cite{ME2016}, some natural conditions for weights depending only on the vertical coordinate on the half-space and for weights radial on the ball are formulated albeit it is not clear if these are sufficient. 

For convenience, we employ the following notation. By $ \bHolAlpha(\Omega, \omega^\alpha) $ we denote weighted holomorphic Bergman space, that is the space of all holomorphic functions on an open set $ \Omega \subset \mathbb{C}^n $ square integrable with respect to the weight function $ \omega^\alpha $, $ \KAlpha $ the corresponding holomorphic Bergman kernel of $ \bHolAlpha(\Omega, \omega^\alpha) $, and its holomorphic Berezin transform by $ \BHolAlpha $. Similarly for the harmonic case, let $ \bHarmAlpha(\Omega, \omega^\alpha) $ be the space of all harmonic functions on an open set $\Omega \subset \mathbb{R}^n $ square integrable with respect to the weight function $ \omega^\alpha $, $ \RAlpha $ its corresponding harmonic Bergman kernel  and $ \BHarmAlpha $ the harmonic Berezin transform.

Compared with the holomorphic case, only a handful of explicit formulas or asymptotic expansions for Bergman kernel and Berezin transform in harmonic settings are known, see \cite{Coifman1980, Miao1998, Jevtic1999, Liu2007, Otahalova2008}. The harmonic Fock space $\bHarmAlpha(\mathbb{R}^n, c_{\alpha} e^{-\alpha |x|^2})$, $\alpha > 0$, is considered in \cite{ME2010}; in \cite{PB2014}, the space of harmonic functions on the unit ball $\bHarmAlpha(\mathbb{B}^n, c_{\alpha}(1-|x|^2)^{\alpha})$, $\alpha > -1$, with $c_{\alpha}$ being a normalizing factor so that in both cases the underlying set has measure $1$. In \cite{JJ2013, Koo2006}, the space of harmonic functions on the upper half-space $\bHarmAlpha(\mathbb{H}^n, y^\alpha)$, $\alpha \geq 0$, $y$ being the vertical coordinate, is considered.

In \cite{ME2016}, which was the main source of inspiration, two of the above mentioned cases are considered in more general settings. For weight function being radial on the unit ball, $\omega(x) = \rho(|x|)$, $x \in \mathbb{B}^n$, and weight function being function of only vertical coordinate on the upper half-space, $\omega(x, y) = \rho(y)$, $(x, y) \in \mathbb{H}^n$, the asymptotic behaviour of the corresponding harmonic Bergman kernels on the diagonal with respect to the weight function $\omega^\alpha$ as $\alpha$ tends to infinity  is given. The main result for the half-space in \cite{ME2016} is the following (for the purpose of this paper, the definition below was extracted from the main theorem):
\begin{definition*}
A positive function $\rho$ on $[0, \infty)$ is \emph{suitable} if it satisfies the following conditions:
\begin{itemize}
  \item $\rho$ makes $\omega$ a defining function, i.e. $\rho(0) = 0$ and $\rho'(0) > 0$;
  \item $\rho$ does not decay too rapidly, i.e. $\int_0^\infty e^{ty} \rho(y)\ dy = +\infty$ for all $t>0$;
  \item $\rho' > 0$ and $(\rho' / \rho)' < 0$ on $(0, \infty)$;
  \item and finally $\rho$ is smooth.
\end{itemize}
\end{definition*}

\begin{theorem*}[\cite{ME2016}, Theorem 3]
Let $\Omega=\mathbb{H}^n$, $n \geq 2$, and assume that $\omega(x, y) = \rho (y)$ depends only on the vertical coordinate and $\rho$ is suitable. Then
\begin{equation*}
\lim_{\alpha \to \infty} \alpha^{1-n} \rho(y)^\alpha \RAlpha(x, y; x, y) = \frac{2^{3-2n}}{\pi^{\frac{n-1}{2}} \Gamma\left( \frac{n-1}{2} \right)} \left( \frac{\rho'}{\rho} \right)^{n-2} \left( -\frac{\rho'}{\rho} \right)'.
\end{equation*}
\end{theorem*}

We extend this result and provide full asymptotic expansion of the weighted harmonic Bergman kernel on the half-space for weights depending only on the vertical coordinate, as well as full asymptotic expansion of the harmonic Berezin transform for functions depending also only on the vertical coordinate. Both expansions are found utilizing the known holomorphic case for the \emph{Siegel domain} $\mathcal{S}^{n-1} = \{ (z, x + iy) \in \mathbb{C}^{n-2} \times \mathbb{C} : y > |z|^2 \}$.

Before we state the results, we fix the following notation. The symbol $\approx$ stands for the usual Poincar\'e asymptotic expansion, i.e. $f(\alpha) \approx \sum_{i=0}^\infty c_i \alpha^{-i}$ if and only if for all $N=0, 1, ...$ we have $f(\alpha) - \sum_{i=0}^{N-1} c_i \alpha^{-i} = O(\alpha^{-N})$ as $\alpha$ tends to infinity. By $\omega_n$ we denote the area of the unit sphere $\mathbb{S}^{n-1}$ in $\mathbb{R}^{n}$, i.e. $\omega_n = \frac{2\pi^{n/2}}{\Gamma\left(n/2\right)}.$ This shall cause no confusion with weight function $\omega$ which is always without the subscript. Also, whenever we mention the function $\rho$, we always assume $\rho$ is suitable. And finally we introduce operator
$$ I_t F = 2^{2-2n}\omega_{n-2} \int_{-1}^{1} (1-t^2)^{\frac{n-4}{2}} F(t)\  dt,\qquad n>2,$$
and given suitable $\rho$ we define function
$$ Q = \left( \frac{\rho'}{\rho} \right)^{n-2} \left(-\frac{\rho'}{\rho} \right)'.$$

\begin{theorem}\label{th:BergmanKernel}
For $\bHarmAlpha(\mathbb{H}^n, \omega^\alpha)$, $\alpha \geq 0$, $ \omega(x, y) = \rho(y) $, with $\rho$ being suitable, the asymptotic expansion of harmonic Bergman kernel $\RAlpha$ as $\alpha$ tends to infinity is given by 
\begin{align}
\intertext{case $n>2$:}
\label{eq:ThRAlphaNg2} \RAlpha(x, y; a, b) &\approx \frac{\alpha^{n-1}}{4\pi^{n-1}} I_t \rho(T)^{-\alpha} Q(T) \sum_{j=0}^{\infty} \frac{c_j(T)}{\alpha^j}
\intertext{evaluated at $T=\displaystyle\frac{y+b}{2} + \frac{|x-a|t}{2i}$,}
\intertext{case $n=2$:}
\label{eq:ThRAlphaNe2} \RAlpha(x, y; a, b) &\approx \frac{\alpha^{n-1}}{4\pi^{n-1}} \left(\rho(T)^{-\alpha} Q(T) \sum_{j=0}^{\infty} \frac{c_j(T)}{\alpha^j} + \overline{\rho(T)^{-\alpha}}\ \overline{Q(T)} \sum_{j=0}^{\infty} \frac{\overline{c_j(T)}}{\alpha^j}\right)
\end{align}
evaluated at $T=\displaystyle\frac{y+b}{2} + \frac{x-a}{2i}$.
\\For both cases, coefficients $c_j \in C^{\infty}(\mathcal{S}^{n-1} \times \mathcal{S}^{n-1})$  are almost analytic at $(z, w) = (Z, W)$ with $c_0 \equiv 1$.
\end{theorem}
Theorem~\ref{th:BergmanKernel} requires some clarifications. Function $Q$ as well as suitable function $\rho$ are defined on the non-negative part of real axis while we evaluate them in (\ref{eq:ThRAlphaNg2}) and (\ref{eq:ThRAlphaNe2}) in right half-plane of the complex plane. We denoted by the same letter $\rho$ its fixed almost-analytic extension to a neighbourhood of the positive real half-axis. For more details, see the proof of the theorem.

\begin{theorem}\label{th:BerezinTransform}
For $f \in L^\infty(\mathbb{H}^n) \cap C^\infty(\mathbb{H}^n)$ depending only on the vertical coordinate, i.e. $f(x, y) = g(y)$ for some function $g$ which is bounded and smooth on $(0, \infty)$, the harmonic Berezin transform $\BHarmAlpha$ for $ \bHarmAlpha(\mathbb{H}^n, \omega^\alpha) $, $\alpha \geq 0$, $\omega(x, y) = \rho(y)$, $\rho$ being suitable, is given by
\begin{equation*}
(\BHarmAlpha f)(a, b) = \sum_{j=0}^{\infty} Q_j f(a, b),
\end{equation*}
where $Q_j$ are cerain linear differential operators whose coefficients involve only derivatives of $\rho$ at the point $b$. The first three terms are given by
\begin{align*}
Q_0 &= I, \\
Q_1 &= \tilde{\Delta} = \frac{4}{\phi''(b)} \frac{\partial^2}{\partial\bar{z}_{n-1} \partial z_{n-1}} + \sum_{k=1}^{n-2} \left( -\frac{1}{\phi'(b)} \right) \frac{\partial^2}{\partial\bar{z}_k \partial z_k} \\
Q_2 &= \frac{1}{2} \tilde{\Delta}^2 + 2 \frac{\psi''(b)}{\phi''(b)^2} \frac{\partial^2}{\partial\bar{z}_{n-1} \partial z_{n-1}} + \frac{1}{2} \sum_{j=1}^{n-2} -\frac{\psi'(b)}{\phi'(b)^2} \frac{\partial^2}{\partial \bar{z}_j \partial z_j},
\end{align*}
where $I$ is identity operator, $\phi = -\log(\rho)$, and $\psi = \log\left(\frac{\phi''}{4} (-\phi')^{n-2}\right)$.
\end{theorem}

In contrast to Theorem~\ref{th:BergmanKernel}, the harmonic Berezin transform in Theorem~\ref{th:BerezinTransform} takes one form regardless of the dimension of $\mathbb{H}^n$ due to the properties of the harmonic Bergman kernel on the diagonal stated explicitly in Corollary~\ref{corollary:RAlphaOnDiagonal}  below.

The rest of the paper is structured in the following way: Section~\ref{sec:BergmanKernel} contains the proof of Theorem~\ref{th:BergmanKernel}, Section~\ref{sec:BerezinTransform} contains the proof of Theorem~\ref{th:BerezinTransform}, and in Section~\ref{sec:Verification}, we recover some results from \cite{JJ2013} by applying Theorem~\ref{th:BergmanKernel} and Theorem~\ref{th:BerezinTransform}.

\section{Bergman kernel}\label{sec:BergmanKernel}

\begin{lemma}\label{lemma:BergamanKernelForSiegelDom}
For the \emph{Siegel domain} $ \mathcal{S}^{n-1} = \{ (z, x + iy) \in \mathbb{C}^{n-2} \times \mathbb{C} : y > |z|^2 \} $ and the weight $ \omega(z, x + iy) = \rho(y - |z|^2) $, the Bergman kernel of $ \bHolAlpha(\mathcal{S}^{n-1}, \omega^\alpha) $ is given by
\begin{equation}\label{eq:KAlpha_Lemma}
\KAlpha(z, x + iy; w, a + ib) = \frac{1}{2\pi} \int_0^\infty \left(\frac{2\xi}{\pi}\right)^{n-2} \frac{e^{i(x-a)\xi - (b+y)\xi + 2\xi \langle z, w \rangle}}{\rhoTildeAlpha(\xi)}\ d\xi,
\end{equation}
where $\rhoTildeAlpha(t) = \int_0^\infty \rho(y)^\alpha\ e^{-2ty}\ dy$. Moreover, the Fourier transform of the kernel $\KAlpha$ with respect to the variable $x$ is supported on $\xi > 0$ and takes the form
\begin{equation}\label{eq:FourierKAlpha_Lemma}
\FourierKAlpha(\xi) = \left(\frac{2\xi}{\pi}\right)^{n-2}\frac{e^{-ia\xi-(b+y)\xi+ 2\xi \langle z, w \rangle}}{\rhoTildeAlpha(\xi)}.
\end{equation}
\end{lemma}
\begin{proof}
See \cite[p. 1256-1257]{ME2016}.
\end{proof}

For convenience, we rename the integration variable $\xi$ to $r$ in (\ref{eq:KAlpha_Lemma}) and define $\KZeroAlpha$ to be $\KAlpha$ with $z = w = 0$
\begin{equation}\label{eq:KZeroAlpha_0}
\KZeroAlpha(x+iy; a+ib) = \KAlpha(0, x+iy; 0, a+ib) = \frac{2^{n-3}}{\pi^{n-1}} \int_0^\infty r^{n-2}\  \frac{e^{i(x-a)r - (b+y)r}}{\rhoTildeAlpha(r)}\ dr.
\end{equation}
Further, seting $ a = 0 $ we obtain 
\begin{equation}\label{eq:KZeroAlpha_1}
\KZeroAlpha(x+iy; ib) = \frac{2^{n-3}}{\pi^{n-1}} \int_0^\infty r^{n-2}\  \frac{e^{ixr - (b+y)r}}{\rhoTildeAlpha(r)}\ dr.
\end{equation}
Finally, we set $w = 0$ and $a=0$ in (\ref{eq:FourierKAlpha_Lemma}) to define
\begin{equation}\label{eq:FourierKZeroAlpha_1}
\FourierKZeroAlpha(\xi) = \left({}_{(y,0,b,z,0)}\KAlpha\right)\hat{\ }(\xi) = \left(\frac{2\xi}{\pi}\right)^{n-2}\frac{e^{-(b+y)\xi}}{\rhoTildeAlpha(\xi)}.
\end{equation}

\begin{lemma}\label{lemma:BergamanKernelForHalfSpace}
For the half-space $\mathbb{H}^{n} = \{ (x, y) \in \mathbb{R}^{n-1} \times \mathbb{R} : y > 0 \}$ and the weight $ \omega(x, y) = \rho(y) $, the Bergman kernel of $ \bHarmAlpha(\mathbb{H}^n, \omega^\alpha) $ is given by
\begin{equation}\label{eq:RAlpha_1}
\RAlpha(x, y; a, b) = \left(2\pi\right)^{1-n} \int_{\mathbb{R}^{n-1}} \frac{e^{i(x-a)\cdot\xi - (b+y)|\xi|}}{\rhoTildeAlpha(|\xi|)}\ d\xi,
\end{equation}
where again $\rhoTildeAlpha(t) = \int_0^\infty \rho(y)^\alpha\ e^{-2ty}\ dy$. Moreover, the Fourier transform of the kernel $\RAlpha$ with respect to the variable $x$ takes the form
\begin{equation*}
\FourierRAlpha(\xi) = \frac{e^{ia \cdot \xi-(b+y)|\xi|}}{\rhoTildeAlpha(|\xi|)}.
\end{equation*}
\end{lemma}

\begin{proof}
Again, see \cite[p. 1255]{ME2016}.
\end{proof}

In polar coordinates $\xi \mapsto r\xi $, $\xi \in \mathbb{S}^{n-2}$, $d\xi = r^{n-2}\ dr\ d\sigma$, equation~(\ref{eq:RAlpha_1}) becomes
\begin{align}
\nonumber \RAlpha(x, y; a, b) &= \left( {2\pi} \right)^{1-n} \int_0^\infty r^{n-2} \frac{e^{-(b+y)r}}{\rhoTildeAlpha(r)} {\int_{\mathbb{S}^{n-2}} e^{i(x-a)\cdot r\xi}\ d\sigma(\xi)}\ dr \\
&= \frac{2^{2-n}\pi^{\frac{1-n}{2}}}{\Gamma\left(\frac{n-1}{2}\right)} \int_0^\infty r^{n-2} \frac{e^{-(b+y)r}}{\rhoTildeAlpha(r)} \myZeroFOne{\frac{n-1}{2}}{-\frac{r^2|x-a|^2}{4}}\ dr\label{eq:RAlpha_2},
\end{align}
where we have used the following formula while integrating over the sphere
\begin{equation*}
\int_{\mathbb{S}^{m-1}} \xi^\alpha\ d\xi = \begin{cases}
  0, & \text{if some $\alpha_j$ is odd}, \\
  2\frac{\Pi_j\ \Gamma\left(\frac{\alpha_j + 1}{2}\right)}{\Gamma\left(\frac{|\alpha|+m}{2}\right)}, & \text{if all $\alpha_j$ are even},
\end{cases}
\end{equation*}
which gives
\begin{equation}\label{eq:int_over_sphere}
\int_{\mathbb{S}^{n-2}} e^{i(x-a)\cdot r\xi}\ d\sigma(\xi) = \omega_{n-1}\ \myZeroFOne{\frac{n-1}{2}}{-\frac{r^2|x-a|^2}{4}}.
\end{equation}

\begin{proof}(of Theorem~\ref{th:BergmanKernel})
For $n>2$, the integral representation
\begin{equation*}
\myZeroFOne{b}{z} = \frac{\Gamma(b)}{\sqrt{\pi}\Gamma\left(b-\frac{1}{2}\right)} \int_{-1}^{1} \frac{(1-t^2)^{b-\frac{3}{2}}}{e^{2t\sqrt{z}}}\ dt,\quad \Re(b) > 1/2
\end{equation*}
leads to
\begin{equation}\label{eq:ZoreFOne}
\myZeroFOne{\frac{n-1}{2}}{-\frac{r^2|x-a|^2}{4}} = \frac{\Gamma\left(\frac{n-1}{2}\right)}{\sqrt{\pi}\Gamma\left(\frac{n-2}{2}\right)} \int_{-1}^{1} (1-t^2)^{\frac{n-4}{2}} e^{ir|x-a|t}\ dt
\end{equation}
and inserting~(\ref{eq:ZoreFOne}) into~(\ref{eq:RAlpha_2}) yields
\begin{align}
\nonumber \RAlpha(x, y; a, b) &= \frac{2^{2-n}\pi^{-\frac{n}{2}}}{\Gamma\left(\frac{n-2}{2}\right)} \int_{-1}^{1} (1-t^2)^{\frac{n-4}{2}} \int_0^\infty r^{n-2} \frac{e^{i|x-a|tr-(b+y)r}}{\rhoTildeAlpha(r)}\ dr\ dt.
\intertext{Using~(\ref{eq:KZeroAlpha_1}) we obtain}
\label{eq:RAlphaInTermsOfHolomNg2} \RAlpha(x, y; a, b) &= 2^{4-2n} \omega_{n-2} \int_{-1}^{1} (1-t^2)^{\frac{n-4}{2}} \KZeroAlpha(|x-a|t+iy; ib)\ dt.
\end{align}
The conditions $\rho' > 0$ and $(\rho' / \rho)' < 0$ in the definition of suitable function are equivalent to the strict pseudoconvexity of the Siegel domain $\mathcal{S}^{n-1}$, which is shown in \cite[p. 1257]{ME2016}. By \cite[Theorem 1]{ME2002} we have
\begin{equation*}
\KAlpha(z, w; Z, W) \approx \frac{\alpha^{n-1}}{\pi^{n-1} \omega(z, w; Z, W)^{\alpha}} \sum_{j=0}^{\infty} \frac{b_j(z, w; Z, W)}{\alpha^j}
\end{equation*}
near the diagonal $(z, w) = (Z, W)$ with coefficients $b_j \in C^{\infty}(\mathcal{S}^{n-1} \times \mathcal{S}^{n-1})$  almost analytic at $(z, w) = (Z, W)$. Here $\omega(z, w; Z, W)$ is a fixed almost-analytic extension of $\omega(z, x+iy) = \rho(y-|z|^2)$ and the first coefficient $b_0$ equals 
\begin{equation*}
b_0(z, w; Z, W) = \det\left[ \partial \bar{\partial} \log \frac{1}{\omega(z, w; Z, W)} \right].
\end{equation*}
In our settings the above with $\omega(z, x+iy) = \rho(y-|z|^2)$ at $\frac{w-\bar{W}}{2i} - \langle z, Z \rangle$ and $z=Z=0$ yields 
\begin{equation}\label{eq:asym_of_K_alpha_zero}
\KZeroAlpha(w; W) \approx \frac{\alpha^{n-1}}{4\pi^{n-1} \rho^{\alpha}} \left( \frac{\rho'}{\rho} \right)^{n-2} \left(-\frac{\rho'}{\rho} \right)' \sum_{j=0}^{\infty} \frac{c_j}{\alpha^j},
\end{equation}
evaluated at $E = \frac{w-\bar{W}}{2i}$. Here we denote by the same letter $\rho$  its fixed almost-analytic extension to a neighbourhood of the positive real half-axis, and $c_j(E)=b_j(0, w; 0, W)/b_0(0,w;0,W)$ with $c_0 \equiv 1$. Using the operator $I_t$, the function $Q$, and writing $w=x+iy$, $W=a+ib$, we insert (\ref{eq:asym_of_K_alpha_zero}) into (\ref{eq:RAlphaInTermsOfHolomNg2}) and obtain full asymptotic expansion for harmonic case, $n>2$, near the diagonal $(x,y)=(a,b)$ in (\ref{eq:ThRAlphaNg2}) evaluated at $T=\displaystyle\frac{y+b}{2} + \frac{|x-a|t}{2i}$. 

For $n=2$, we use identity $\myZeroFOne{1/2}{-z} = \cos(2\sqrt{z})$ in~(\ref{eq:RAlpha_2}) to obtain
\begin{align*}
\RAlpha(x, y; a, b) &= \frac{1}{\pi} \int_0^\infty \frac{e^{-(b+y)r}}{\rhoTildeAlpha(r)} \myZeroFOne{\frac{1}{2}}{-\left(\frac{r|x-a|}{2}\right)^2}\ dr\\
&= \frac{1}{\pi} \int_0^\infty \frac{e^{-(b+y)r}}{\rhoTildeAlpha(r)} \cos(r|x-a|)\ dr.
\end{align*}
Using formula $\cos(x) = (e^{ix}+e^{-ix})/2$ and the fact that function $\cos(x)$ is even, we continue by
\begin{align*}
\RAlpha(x, y; a, b) &=\frac{1}{2\pi} \left( \int_0^\infty \frac{e^{ir(x-a)}e^{-(b+y)r}}{\rhoTildeAlpha(r)}\ dr + \int_0^\infty \frac{e^{-ir(x-a)}e^{-(b+y)r}}{\rhoTildeAlpha(r)}\ dr\right)\\
&=\frac{1}{2\pi} \left( 2\pi\KZeroAlpha(x+iy;a+ib) + 2\pi\overline{\KZeroAlpha(x+iy;a+ib)}  \right)
\end{align*}
where the second equality comes from~(\ref{eq:KZeroAlpha_0}) for $n=2$ when Siegel domain $\mathcal{S}^1$ reduces to the upper half-plane of the complex plane. Therefore 
\begin{equation}\label{eq:RAlphaInTermsOfHolomNe2}
\RAlpha(x, y; a, b) = 2\Re \KZeroAlpha(x+iy;a+ib),
\end{equation}
and we insert  (\ref{eq:asym_of_K_alpha_zero}) into (\ref{eq:RAlphaInTermsOfHolomNe2}) to obtain (\ref{eq:ThRAlphaNe2}) evaluated at $T=\displaystyle\frac{y+b}{2} + \frac{x-a}{2i}$.
\end{proof}

\begin{corollary}\label{corollary:RAlphaOnDiagonal}
On the diagonal, the harmonic kernel $\RAlpha$ reduces to
\begin{equation*}
\RAlpha(a,b;a,b) = 2^{4-2n} \omega_{n-1} \KZeroAlpha(a+ib;a+ib) = 2^{4-2n} \omega_{n-1} \KZeroAlpha(ib; ib).
\end{equation*}
\end{corollary}

\begin{proof}
Identity $\myZeroFOne{b}{0} = 1$ has no restriction on the parameter $b$ and from equation~(\ref{eq:RAlpha_2}) we get
\begin{equation*}
\RAlpha(a,b;a,b) = \frac{2^{2-n}\pi^{\frac{1-n}{2}}}{\Gamma\left(\frac{n-1}{2}\right)} \int_0^\infty r^{n-2} \frac{e^{-2br}}{\rhoTildeAlpha(r)}\ dr,
\end{equation*}
from equation~(\ref{eq:KZeroAlpha_0}) we have
\begin{equation*}
\KZeroAlpha(a+ib; a+ib) = \frac{2^{n-3}}{\pi^{n-1}} \int_0^\infty r^{n-2}\  \frac{e^{-2br}}{\rhoTildeAlpha(r)}\ dr
\end{equation*}
and the result comes from comparing the above two expressions. 
\end{proof}

\section{Berezin transform}\label{sec:BerezinTransform}

For a bounded function $f$ on $\mathbb{H}^n$, the harmonic Berezin transform of  $ \bHarmAlpha(\mathbb{H}^n, \omega^\alpha) $, $ \omega(x, y) = \rho(y) $, is given by equation~(\ref{eq:harmBerezinTr}), i.e.
\begin{align*}
(\BHarmAlpha f)(a, b) &= \frac{1}{\RAlpha(a, b; a, b)} \frac{2^{10-4n} \pi^{n-2}}{\left(\Gamma\left(\frac{n-2}{2}\right)\right)^2} \int_{\mathbb{R}^{n-1}} \int_0^\infty f(x, y) \rho(y)^\alpha \\
&\times \int_{-1}^{1} (1-t^2)^{\frac{n-4}{2}} \KZeroAlpha(|x-a|t+iy; ib) \\
&\times \int_{-1}^{1} (1-s^2)^{\frac{n-4}{2}} \KZeroAlpha(|x-a|s+iy; ib)\ ds\ dt\ dy\ dx,\quad  &n>2,\\
(\BHarmAlpha f)(a, b) &= \frac{4}{\RAlpha(a, b; a, b)} \int_{\mathbb{R}} \int_0^\infty f(x, y) \rho(y)^\alpha \\
&\times \left( \Re(\KZeroAlpha(x+iy;a+ib)) \right)^2\ dy\ dx \quad &n=2.
\end{align*}
Suppose $f$ depends only on the vertical coordinate, that is $f(x,y) = g(y)$ for some bounded function $g$ on $(0, \infty)$. We define a bounded function $h$ on the Siegel domain in terms of $f$ by $h(z, x+iy) = g(y-|z|^2)$, that is function $h$ is of the same form as the weight on the Siegel domain under the consideration and it is clear that $h(0,iy) = f(0,y)$.


\begin{proof}(of Theorem~\ref{th:BerezinTransform})
The Berezin transform for a function of the from $h(z, x+iy) = g(y-|z|^2) $ on $\bHolAlpha(\mathcal{S}^{n-1}, \omega^\alpha)$ is given by
\begin{align} \label{eq:BAlphaHol}
(\BHolAlpha h)(0, ib) &= \frac{1}{\KZeroAlpha(ib; ib)} \\
\nonumber &\times \int_{\mathbb{C}^{n-2}} \int_{|z|^2}^{\infty} g(y-|z|^2) \rho(y-|z|^2)^\alpha \\
\nonumber &\times \int_{\mathbb{R}}  |\KAlpha(z, x+iy; 0, ib)|^2\ dx\ dy\ dz.
\end{align}
Using Parseval's identity, equation (\ref{eq:FourierKZeroAlpha_1}), and Lemma~\ref{lemma:BergamanKernelForSiegelDom} we compute
\begin{align*}
\int_{\mathbb{R}}  |\KAlpha(z, x+iy; 0, ib)|^2\ dx &=
\frac{1}{2\pi} \int_{\mathbb{R}}  |\FourierKZeroAlpha(\xi)|^2\ d\xi \\
&= \frac{1}{2\pi} \int_{\mathbb{R}}
\left(\frac{2\xi}{\pi}\right)^{n-2}\frac{e^{-(b+y)\xi}}{\rhoTildeAlpha(\xi)}
\left(\frac{2\xi}{\pi}\right)^{n-2}\frac{e^{-(b+y)\xi}}{\rhoTildeAlpha(\xi)}
\ d\xi \\
&= 2^{2n-5} \pi^{3-2n} \int_{0}^{\infty} \xi^{2n-4}\frac{e^{-2(b+y)\xi}}{\rhoTildeAlpha(\xi)^2}
\ d\xi,
\end{align*}
and insert the above into equation~(\ref{eq:BAlphaHol}) while renaming the integration variable $\xi$ to $r$ at the same time. We obtain
\begin{align*}
(\BHolAlpha h)(0, ib) &= \frac{2^{2n-5}\pi^{3-2n}}{\KZeroAlpha(ib; ib)} \int_{\mathbb{C}^{n-2}} \int_{|z|^2}^{\infty} g(y-|z|^2) \rho(y-|z|^2)^\alpha \\
&\times \int_0^\infty r^{2n-4} \frac{e^{-2(b+y)r}}{\rhoTildeAlpha(r)^2}\ dr\ dy\ dz.
\end{align*}
Change of variables $ y \mapsto y-|z|^2 $ and re-arranging the integrals leads to
\begin{align*}
(\BHolAlpha h)(0, ib) &= \frac{2^{2n-5}\pi^{3-2n}}{\KZeroAlpha(ib; ib)} \int_{0}^{\infty} \int_{0}^{\infty} g(y) \rho(y)^\alpha r^{2n-4} \frac{e^{-2(b+y)r}}{\rhoTildeAlpha(r)^2} \\
&\times \int_{\mathbb{C}^{n-2}} e^{-2r|z|^2}\ dz\ dr\ dy.
\end{align*}
Here $\int_{\mathbb{C}^{n-2}} e^{-2r|z|^2}\ dz$ is the Gaussian integral over $\mathbb{C}^{n-2}$ and equals to $\left( \frac{\pi}{2r} \right)^{n-2}$, hence
\begin{align*}
(\BHolAlpha h)(0, ib) = \frac{2^{n-3}\pi^{1-n}}{\KZeroAlpha(ib; ib)} \int_{0}^{\infty} \int_{0}^{\infty} g(y) \rho(y)^\alpha r^{n-2} \frac{e^{-2(b+y)r}}{\rhoTildeAlpha(r)^2}\ dr\ dy,
\end{align*}
and as a matter of convenience, we express the double integral as
\begin{equation}\label{eq:holomorphic_integral}
\int_{0}^{\infty} \int_{0}^{\infty} g(y) \rho(y)^\alpha r^{n-2} \frac{e^{-2(b+y)r}}{\rhoTildeAlpha(r)^2}\ dr\ dy = \frac{\KZeroAlpha(ib; ib)}{2^{n-3}\pi^{1-n}} (\BHolAlpha h)(0, ib).
\end{equation}

For the harmonic case, the Berezin transform of a function $f$ depending only on the vertical coordinate, that is $f(x, y) = g(y)$, takes the form 
\begin{equation}\label{eq:BAlphaHarm}
(\BHarmAlpha f)(a, b) = \frac{1}{\RAlpha(a, b; a, b)} \int_0^\infty g(y) \rho(y)^\alpha \int_{\mathbb{R}^{n-1}} |\RAlpha(x, y; a, b)|^2\ dx\ dy.
\end{equation}
Using Parseval's identity again and Lemma~\ref{lemma:BergamanKernelForHalfSpace} we have
\begin{align*}
\int_{\mathbb{R}^{n-1}} |\RAlpha(x, y; a, b)|^2\ dx &= \frac{1}{(2\pi)^{n-1}} \int_{\mathbb{R}^{n-1}} |\FourierRAlpha(\xi)|^2\ d\xi \\
&= \frac{1}{(2\pi)^{n-1}} \int_{\mathbb{R}^{n-1}} \frac{e^{ia \cdot \xi-(b+y)|\xi|}}{\rhoTildeAlpha(|\xi|)} \frac{e^{-ia \cdot \xi-(b+y)|\xi|}}{\rhoTildeAlpha(|\xi|)}\ d\xi \\
&= \frac{1}{(2\pi)^{n-1}} \int_{\mathbb{R}^{n-1}} \frac{e^{-2(b+y)|\xi|}}{\rhoTildeAlpha(|\xi|)^2} \ d\xi \quad \text{(to polar coordinates $\xi \mapsto r\xi$)}\\
&= \frac{1}{(2\pi)^{n-1}} \int_0^\infty r^{n-2} \frac{e^{-2(b+y)r}}{\rhoTildeAlpha(r)^2} \int_{\mathbb{S}^{n-2}} \ d\sigma(\xi)\ dr \\
&= 2^{1-n} \pi^{1-n} \omega_{n-1} \int_0^\infty r^{n-2} \frac{e^{-2(b+y)r}}{\rhoTildeAlpha(r)^2} \ dr.
\end{align*}
Inserting the above into~(\ref{eq:BAlphaHarm}) we get
\begin{align*}
(\BHarmAlpha f)(a, b) &= \frac{2^{1-n} \pi^{1-n} \omega_{n-1}}{\RAlpha(a, b; a, b)} \int_0^\infty \int_0^\infty g(y) \rho(y)^\alpha r^{n-2} \frac{e^{-2(b+y)r}}{\rhoTildeAlpha(r)^2} \ dr\ dy \label{eq:harmonic_integral}\\
&= \frac{2^{1-n} \pi^{1-n} \omega_{n-1}}{\RAlpha(a, b; a, b)}\ \frac{\KZeroAlpha(ib; ib)}{2^{n-3}\pi^{1-n}} (\BHolAlpha h)(0, ib) \\
&=  \frac{2^{1-n} \pi^{1-n} \omega_{n-1}}{2^{4-2n}\omega_{n-1}\KZeroAlpha(a+ib;a+ib)}
    \frac{\KZeroAlpha(ib; ib)}{2^{n-3}\pi^{1-n}}\ \BHolAlpha f(0, ib) \\
&=(\BHolAlpha h)(0, ib),
\end{align*}
where we have used~(\ref{eq:holomorphic_integral}) in the second equality and Corollary~\ref{corollary:RAlphaOnDiagonal} to finish the calculation.

To obtain full asymptotic expansion of $\BHarmAlpha$ as $\alpha \to \infty$ in terms of $\BHolAlpha$, we use \cite[Theorem 2]{ME2002}
\begin{equation*}
(\BHolAlpha f)(0, ib) = \sum_{j=0}^{\infty} Q_j f(0, ib),
\end{equation*}
where $Q_j$ are linear differential operators whose coefficients involve only derivatives of $\rho$ at the point $(0, ib)$. The first three terms are given by \cite[Theorem 11]{ME2002}, i.e. 
\begin{align*}
Q_0 &= I, \\
Q_1 &= \tilde{\Delta} \equiv \sum_{j,k} g^{\bar{j}k} \frac{\partial^2}{\partial\bar{z}_j \partial z_k}, \\
Q_2 &= \frac{1}{2} \tilde{\Delta}^2 + \frac{1}{2} \sum_{j,k} {\Ric}^{\bar{j}k} \frac{\partial^2}{\partial\bar{z}_j \partial z_k} , \\
\end{align*}
where $I$ is identity operator, $g^{\bar{j}k}$ is inverse matrix to $g_{j\bar{k}} = \partial^2 (-\log \rho)/\partial z_j \partial \bar{z}_k$, and ${\Ric}^{\bar{j}k}$ is contravariant Ricci tensor
\begin{equation*}
{\Ric}^{\bar{j}k} = \sum_{l,m} g^{\bar{j}l} g^{\bar{m}k} \frac{\partial^2 \log \chi}{\partial \bar{z}_m \partial z_l}
\end{equation*}
with $\chi = \det[g_{j\bar{k}}]$. 

The matrix $g_{j\bar{k}}$ was computed in \cite{ME2016} on p. 1257 in terms of the function $\phi = -\log(\rho)$, and evaluation at $(0, ib)$ yields
\begin{subequations}
\begin{align}\label{eq:g_jk}
g_{j\bar{k}} &= \begin{bmatrix}
-\phi'(b)I & 0 \\
0          & \phi''(b)/4  
\end{bmatrix},\\
\label{eq:inverse_g_jk}
g^{\bar{j}k} &= \begin{bmatrix}
-1/\phi'(b)\ I & 0 \\
0              & 4/\phi''(b)  
\end{bmatrix},
\end{align}
\end{subequations}
where $I$ is the identity matrix of size $n-2$, $\phi'' = \left(-\frac{\rho'}{\rho} \right)'$, and $\phi' = -\frac{\rho'}{\rho}$, which gives
\begin{equation}\label{eq:Q1}
Q_1 = \tilde{\Delta} = \frac{4}{\phi''(b)} \frac{\partial^2}{\partial\bar{z}_{n-1} \partial z_{n-1}} + \sum_{k=1}^{n-2} \left( -\frac{1}{\phi'(b)} \right) \frac{\partial^2}{\partial\bar{z}_k \partial z_k}.
\end{equation}
Diagonal matrix (\ref{eq:inverse_g_jk}) implies that ${\Ric}^{\bar{j}k} = 0$ for $j \neq k$ so we are left only with entries for $j=k$ 
\begin{equation*}
{\Ric}^{\bar{j}j} = \left( g^{\bar{j}j} \right)^2\ \frac{\partial^2 \log \chi}{\partial \bar{z}_j \partial z_j}.
\end{equation*}
The second partial derivatives of the function 
\begin{equation*}
\psi = \log \chi = \log \det[g_{j\bar{k}}] = \log\left(\frac{\phi''}{4} (-\phi')^{n-2}\right)
\end{equation*}
evaluated at $(0, ib)$ are given by (\ref{eq:g_jk}) by replacing $\phi$ with $\psi$, hence $Q_2$ takes the form
\begin{equation*}
Q_2 = \frac{1}{2} \tilde{\Delta}^2 + 2 \frac{\psi''(b)}{\phi''(b)^2} \frac{\partial^2}{\partial\bar{z}_{n-1} \partial z_{n-1}} + \frac{1}{2} \sum_{j=1}^{n-2} -\frac{\psi'(b)}{\phi'(b)^2} \frac{\partial^2}{\partial \bar{z}_j \partial z_j}.
\end{equation*}
\end{proof}

\section{Verification}\label{sec:Verification}

Firstly, we recover the result for the harmonic Bergman kernel stated in \cite{JJ2013} by applying Theorem~\ref{th:BergmanKernel} and using well known holomorphic kernel on the Siegel domain for case $n>2$.

The harmonic Bergman kernel for $\bHarmAlpha(\mathbb{H}^n, \omega^\alpha)$, $\omega(x, y) = y$, is given by equation (16) in \cite{JJ2013}. Using polar coordinates, see \cite[page 725]{JJ2013}, we compute
\begin{align}
\nonumber \RAlpha(x, y; 0, 1) &= \frac{2^{\alpha+1}}{(2\pi)^{n-1} \Gamma(\alpha+1)} \int_{0}^{\infty} r^{\alpha+n-1} e^{-(y+1)r} \int_{\mathbb{S}^{n-2}} e^{ir\xi \cdot x} d\sigma(\xi) dr \\
\nonumber &= \frac{2^{\alpha+1} \omega_{n-1}}{(2\pi)^{n-1} \Gamma(\alpha+1)} \int_{0}^{\infty} r^{\alpha+n-1} e^{-(y+1)r} \myZeroFOne{\frac{n-1}{2}}{-\frac{r^2|x|^2}{4}} dr \\
\nonumber &= \frac{2^{\alpha-n+2} \omega_{n-1} \Gamma\left(\frac{n-1}{2}\right)}{\pi^{n-\frac{1}{2}} \Gamma(\alpha+1) \Gamma\left(\frac{n-2}{2}\right)} \int_{-1}^{1} \left(1-t^2\right)^{\frac{n-4}{2}} \int_{0}^{\infty} r^{\alpha+n-1} e^{-(y+1-i|x|t)r} dr\ dt \\
\label{eq:RAlphaOnHNFor0_1}&= \frac{2^{\alpha-n+3} \Gamma(\alpha+n)}{\pi^{\frac{n}{2}} \Gamma(\alpha+1) \Gamma\left(\frac{n-2}{2}\right)} \int_{-1}^{1} \left(1-t^2\right)^{\frac{n-4}{2}}  \left(y+1-i|x|t\right)^{-\alpha-n}\ dt,
\end{align} 
where we used equation~(\ref{eq:int_over_sphere}) to integrate over the sphere in the second equality, integral representation~(\ref{eq:ZoreFOne}) for $\myZeroFOne{b}{z}$ in the third equality, and finally evaluated integral over $dr$ using equality
\begin{equation*}
\int_{0}^{\infty} r^{\alpha+n-1} e^{-zr} dr = \Gamma(\alpha+n)z^{-\alpha-n},\quad \Re z > 0.
\end{equation*}

The holomorphic Bergman kernel for $\bHolAlpha(\mathcal{S}^{n-1}, \omega^\alpha)$, $\omega(z, x+iy) = y-|z|^2$, on the Siegel domain is well know to be
\begin{equation*}
\KAlpha(z, x+iy; w, a+ib) = \frac{2^{n+\alpha-2} \Gamma(n+\alpha)}{\pi^{n-1} \Gamma(\alpha+1)} \left( i(\overline{a+ib} - (x+iy)) - 2\langle z, w \rangle \right)^{-n-\alpha}.
\end{equation*}
Specializing the above for $z=w=0$ and $a=0$ we get
\begin{equation}\label{eq:KZeroAlphaForKnownSiegelCase}
\KZeroAlpha(x+iy; ib) = \frac{2^{n+\alpha-2} \Gamma(n+\alpha)}{\pi^{n-1} \Gamma(\alpha+1)} \left( b - ix + y \right)^{-n-\alpha}.
\end{equation}
Inserting~(\ref{eq:KZeroAlphaForKnownSiegelCase}) into~(\ref{eq:RAlphaInTermsOfHolomNg2}) and evaluating at $(x, y; 0, 1)$ gives
\begin{align*}
\RAlpha(x, y; 0, 1) &= 2^{4-2n}\omega_{n-2} \int_{-1}^{1} (1-t^2)^{\frac{n-4}{2}} \KZeroAlpha(|x|t+iy; i)\ dt \\
&= \frac{2^{\alpha-n+3} \Gamma(\alpha+n)}{\pi^{\frac{n}{2}} \Gamma(\alpha+1) \Gamma\left(\frac{n-2}{2}\right)} \int_{-1}^{1} \left( 1-t^2 \right)^{\frac{n-4}{2}} (1+y-i|x|t)^{-n-\alpha}\ dt,
\end{align*}
which is~(\ref{eq:RAlphaOnHNFor0_1}).

Secondly, we recover the second differential operator $R_1$ in the harmonic Berezin transform in \cite[page 729]{JJ2013}. We have
\begin{equation*}
R_1 f (x,y) = y^2\frac{\Delta f}{n-1} (x,y) + (2-n) y \frac{\partial f}{\partial y} (x,y) + y^2 \frac{\partial^2 f}{\partial y^2} (x,y)
\end{equation*}
which for the function depending only on the vertical coordinate reduces to
\begin{equation}\label{eq:R1}
R_1 f (y) = (2-n) y f'(y) + y^2 f''(y).
\end{equation}
The second partial derivatives of the function $f(y-|z|^2)$ evaluated at $z=0$ are given by (\ref{eq:g_jk}) by replacing $\phi$ with $f$. Also
\begin{equation*}
\phi(y) = -\log(y),\quad \phi'(y) = -\frac{1}{y},\quad \phi''(y) = \frac{1}{y^2},
\end{equation*}
hence inserting the above into~(\ref{eq:Q1}) gives~(\ref{eq:R1})
\begin{align*}
Q_1 f(y) = y^2 f''(y) + \sum_{k=1}^{n-2} y (-f'(y)) = y^2 f''(y) + (2-n)y  f'(y).
\end{align*}

\ifextras
\section*{Acknowledgements}
The author thanks Miroslav Engli\v{s} for the introduction to the subject.
\fi

\bibliographystyle{plain}
\bibliography{HarmBerezinransformSpecialCase}

\end{document}